\documentclass{article}
\usepackage[pdftex]{hyperref}
\usepackage{tikz-cd, url}

\hypersetup{
  colorlinks   = true, 
  urlcolor     = gray, %Color for external hyperlinks
  linkcolor    = red, %Color of internal links
  citecolor   =  blue%Color of citations
}
\usepackage{amsmath,amsfonts,amsthm,amssymb,graphicx,xcolor}
\usepackage[all,2cell,ps]{xy}

\theoremstyle{plain}
\newtheorem{thm}{Theorem}[section]
\newtheorem{lem}[thm]{Lemma}
\newtheorem{prop}[thm]{Proposition}
\newtheorem{cor}[thm]{Corollary}
\newtheorem{conj}[thm]{Conjecture}

\theoremstyle{definition}

\newtheorem{qtn}[thm]{Question}
\newtheorem{rem}[thm]{Remark}

\newcommand{\Z}{\mathbb Z}

\newcommand\sO{{\mathcal O}}

\DeclareMathOperator{\Alb}{Alb}

\usepackage{mathtools}

\DeclarePairedDelimiter\abs{\lvert}{\rvert}
\DeclarePairedDelimiter\norm{\lVert}{\rVert}

\makeatletter
\let\oldabs\abs
\def\abs{\@ifstar{\oldabs}{\oldabs*}}
\let\oldnorm\norm
\def\norm{\@ifstar{\oldnorm}{\oldnorm*}}
\makeatother
\newenvironment{pf}{\begin{proof}}{\end{proof}}

\title{Aspherical Complex Surfaces, the Singer Conjecture, and Gromov-L\"uck Inequality $\chi\geq |\sigma|$}
 \author{\small{Michael Albanese}\\ \scriptsize{University of Waterloo} \\ \footnotesize{\textsf{m3albanese@uwaterloo.ca}} \and \small{Luca F. Di Cerbo}\footnote{Supported in part by NSF grant DMS-2104662} \\ \scriptsize{University of Florida} \\ \footnotesize{\textsf{ldicerbo@ufl.edu}} \and 
\small{Luigi Lombardi}\footnote{Partially supported by GNSAGA of INDAM, PRIN 2020: Curves, Ricci flat Varieties and their Interactions} \\ 
\scriptsize{Università degli Studi di Milano Statale}\\ \footnotesize{\textsf{luigi.lombardi@unimi.it}}}
\date{}
	
\begin{document}

\maketitle

%%%%%%%%%%%%%%%%%%%%
\begin{abstract}
We discuss the Singer conjecture and Gromov-L\"uck inequality $\chi\geq |\sigma|$ for aspherical complex surfaces. We give a proof of the Singer conjecture for aspherical complex surface with residually finite fundamental group that does not rely on Gromov's K\"ahler groups theory. Without the residually finiteness assumption, we observe that this conjecture can be proven for all aspherical complex surfaces except possibly those in Class $\mathrm{VII}_0^+$ (a positive answer to the global spherical shell conjecture would rule out the existence of aspherical surfaces in this class). We also sharpen Gromov-L\"uck inequality for aspherical complex surfaces that are not in Class $\mathrm{VII}_0^+$. This is achieved by connecting the circle of ideas of the Singer conjecture with the study of Reid's conjecture. 
\end{abstract}
%%%%%%%%%%%%%%%%%%%%
\vspace{8cm}
\tableofcontents\quad\\

\vspace{1cm}

%%%%%%%%%%%%%%%%%%%%%%%%%%%%%%%%
\section{Introduction and Main Results}
%%%%%%%%%%%%%%%%%%%%%%%%%%%%%%%%

A significant part of modern Riemannian geometry deals with the interaction between curvature and topology of smooth manifolds. As beautifully recounted in Marcel Berger's panoramic book on Riemannian geometry (see in particular Chapter 12 in \cite{Berger}), Heinz Hopf was the first to investigate the connections between topology and curvature in a general and systematic way.  Surprisingly, some of the questions that Hopf posed in the 1930's remain unanswered. A well-known example is the following problem on the sign of the Euler characteristic of aspherical manifolds. 

\begin{conj}[Hopf Conjecture]\label{Hopf}
	If $X$ is a closed aspherical manifold of dimension $2n$, then:
	\begin{equation*}
	(-1)^n\chi_{top}(X)\geq 0.
	\end{equation*}
	
\end{conj}

Thanks to the uniformization theorem for Riemann surfaces, Conjecture \ref{Hopf} is true when $n=1$. On the other hand, this problem is still completely open when $n=2$, but if $X^4$ is a closed, non-positively curved $4$-manifold, then John Milnor proved that Conjecture \ref{Hopf} is indeed true in this case, see \cite[Chapter 12, Note 12.3.1.1]{Berger}. There are several families of closed aspherical $4$-manifolds which do not admit non-positively curved metrics. For example, non-flat nilmanifolds cannot admit such a metric by \cite[Corollary A]{Yau71} -- of course, these have Euler characteristic zero, so the Hopf conjecture holds for them.

During the 1970's, Isadore Singer suggested an approach to Conjecture \ref{Hopf} via the study of $L^2$-harmonic forms on the topological universal cover of $X$. Taking into account Atiyah's $L^2$-index theorem \cite{Atiyah}, he proposed the following.

\begin{conj}[Singer Conjecture]\label{Singer}
	If $X$ is a closed aspherical manifold of real dimension $2n$, then the $L^2$-Betti numbers are:
	\begin{equation*}
	b^{(2)}_{k}(X; \widetilde{X})=\begin{cases} 
	(-1)^n \chi_{top}(X)  &\mbox{if }\quad    k =  n  \\
	0 & \mbox{if }\quad   k  \neq n
	\end{cases}
	\end{equation*}
	where $\pi \colon \widetilde{X}\rightarrow X$ is the topological universal cover of $X$.
\end{conj}

%Indeed, he observed that an affirmative solution to Conjecture \ref{Singer} would also settle Conjecture \ref{Hopf} via Atiyah's $L^2$-index theorem \cite{Atiyah}. 

An affirmative solution to Conjecture \ref{Singer} would also settle Conjecture \ref{Hopf}. For more details on this circle of ideas, we refer to Shing-Tung Yau's influential list of main open problems in geometry \cite[Section VII, Problem 10]{SchoenY}. We also refer to Wolfgang L\"uck's book \cite{Luck02} for the definition of $L^2$-Betti numbers and for a comprehensive account on the history of the Singer conjecture. Interestingly, Conjecture \ref{Singer} is not known to be true for $n=2$ even under the assumption that $X^4$ is non-positively curved.

As observed and discussed by Mikhael Gromov in \cite[Section 8]{Gro93} and Wolfgang L\"uck \cite[Theorem 5.1]{Luck1}, Conjecture \ref{Singer} implies an \emph{effective} version of Conjecture \ref{Hopf} in dimension four. More precisely, one can state the following intriguing conjecture regarding the geography of aspherical $4$-manifolds. 

\begin{conj}[Gromov-L\"uck Inequality]\label{Gromov}
	If $X$ is a closed, oriented, aspherical $4$-manifold, then:
	\begin{equation*}
	\chi_{top}(X)\geq |\sigma(X)|,
	\end{equation*}
where $\sigma(X)$ is the signature of $X$.
\end{conj}

In this paper, we study Conjectures \ref{Hopf}, \ref{Singer}, and \ref{Gromov} on closed, aspherical $4$-manifolds that admit a complex structure. Our knowledge of compact complex surfaces via the Kodaira-Enriques classification is a powerful tool in this case. For example in \cite[Theorem 2]{JK93}, Johnson-Kotschick show that any complex surface $X$ satisfies the inequality $\chi_{top}\geq |\sigma|$ unless $X$ is a ruled surface over a curve of genus $g\geq 2$. Since ruled surfaces are \emph{not} aspherical, we therefore have that Conjecture \ref{Gromov} (and then also Conjecture \ref{Hopf}) is true for aspherical complex surfaces. With that said, we do believe that a deeper study of Conjecture \ref{Gromov} is warranted even for complex surfaces. First, Conjecture \ref{Gromov} is quite crude when compared with other geometric inequalities constraining the geography of vast classes of $4$-manifolds. For example, the equality cases in the Hitchin-Thorpe inequality for Einstein $4$-manifolds and the Bogomolov-Miyaoka-Yau inequality for minimal surfaces of general type are neatly characterized. On the other hand, there is no (conjectural) characterization of the equality case in Conjecture \ref{Gromov}. Moreover, one may wonder if it is useful to sharpen Gromov-L\"uck inequality to a tighter constraint on the geography of aspherical $4$-manifolds that do \emph{not} satisfy $\chi_{top}=\sigma=0$. We do have a \emph{quite} satisfactory answer to all such questions when $X$ is an aspherical complex surface.

\begin{thm}\label{Ginequality}
Let $X$ be a closed, aspherical, complex surface. We have the following possibilities for its Euler characteristic $\chi_{top}(X)$ and signature $\sigma(X)$:
	\begin{itemize} 
	\item $\chi_{top}(X)= -\sigma(X)>0$ in which case $X$ is a Class $\mathrm{VII}_0^+$ surface violating the global spherical shell conjecture;
       \item $\chi_{top}(X)\geq \frac{9}{5}|\sigma(X)|$ and $\chi_{top}(X) > 0$ in which case $X$ is of general type;
       \item $\chi_{top}(X)=\sigma(X)=0$ in all other cases. 
	\end{itemize}
\end{thm}

In particular, we see that if the global spherical shell conjecture is true, then Gromov-L\"uck inequality is always strict for closed, aspherical, complex surfaces unless the Euler characteristic and signature are both zero. Furthermore, we obtain a factor $9/5>1$ in front of the absolute value of the signature in all of the remaining cases.

Next, we address the big elephant in the room: is Conjecture \ref{Singer} true for closed, aspherical, complex surfaces? We observe that Gromov's characterization of closed K\"ahler manifolds with non-vanishing first $L^2$-Betti number, when combined with the Kodaira-Enriques classification, suffices to show this conjecture holds true for all closed, aspherical, complex surfaces that are not in Class $\mathrm{VII}_0^+$, see Theorem \ref{csurfaces}. Frustratingly enough, the validity of Conjecture \ref{Singer} also stumbles upon the existence of aspherical surfaces in Class $\mathrm{VII}_0^+$. We conclude by providing a proof of the following.

\begin{thm}\label{main3}
	Let $X$ be a closed, aspherical complex surface with residually finite 
	fundamental group, and let $\widetilde{X}$ be the topological universal cover. The $L^2$-Betti numbers are:
	\begin{equation*}
	b^{(2)}_k (X; \widetilde{X})=\begin{cases} 
	\chi_{top}(X)  &\mbox{if }\quad    k =  2  \\
	0 & \mbox{if }\quad   k  \neq 2
	\end{cases}
	\end{equation*}
\end{thm}

Our proof of Theorem \ref{main3} does not rely upon Gromov's theory of K\"ahler groups. It combines the study of the Albanese map, L\"uck's approximation theorem and the Kodaira-Enriques classification.

In real dimension greater than or equal to four, there is a plethora of examples of closed aspherical manifolds whose fundamental group is not residually finite. Such examples can be constructed with the so-called Davis reflection trick \cite{Davis}, see for example \cite{Mess90}. 
It seems to be currently unknown whether or not there are examples of aspherical smooth projective varieties with non-residually finite $\pi_1$. Indeed, the examples of Toledo \cite{Dom93} and Catanese-Koll\'ar \cite{CK90} of smooth projective varieties with non-residually finite $\pi_1$ appear not to be aspherical. It is currently unknown whether the non-positively curved smooth minimal toroidal compactifications of ball quotients identified in \cite[Theorem A]{tesi}, or the negatively curved branched covers constructed in  \cite[Theorem 1.5]{S22} have residually finite $\pi_1$. 

\noindent\textbf{Acknowledgments}. 
The authors thank Wolfgang L\"uck for for useful bibliographical suggestions and for pertinent comments on the manuscript. The first named author thanks Vestislav Apostolov for answering a question about Class $\mathrm{VII}_0^+$ surfaces. The second named author thanks Fabrizio Catanese, Rita Pardini, Matthew Stover, and Roberto Svaldi for valuable discussions. He also thanks the Mathematics Departments of the University of Milan and the University of Waterloo for the invitation to present research related to this project, for support, and for the nice working environments during his visits in the Spring of 2023. The third named author thanks Alice Garbagnati for useful conversations, and the Mathematics Department of the University of Florida for the optimal working environment provided during his visit in the Spring of 2023. \\ \\

\section{Aspherical Complex Surfaces}\label{asphericalS}

In this section, we give a brief overview of those compact complex surfaces which are aspherical. These surfaces have contractible universal cover or equivalently, $\pi_k$ vanishes for $k > 1$. By \cite[Lemma 2]{ADC}, such surfaces are minimal. We will work through the Kodaira-Enriques classification by Kodaira dimension.

{\bf Kodaira dimension $-\infty$:} In the K\"ahler case, such surfaces are rational or ruled. The former consists of $\mathbb{CP}^2$ and Hirzebruch surfaces $\Sigma_n = \mathbb{P}_{\mathbb{CP}^1}(\mathcal{O}\oplus\mathcal{O}(n))$. These are all simply connected, so they are their own universal covers. As they are not contractible, rational surfaces are not aspherical. Ruled surfaces are holomorphic fiber bundles with fiber $\mathbb{CP}^1$ and structure group $PGL(2, \mathbb{C})$ over a smooth connected curve $C$ of positive genus. Every such surface is the projectivisation of a rank two holomorphic vector bundle over $C$. From the long exact sequence in homotopy, it follows that ruled surfaces have non-zero $\pi_2$. In fact, if $\widetilde{C} \to C$ denotes the universal covering of $C$, pulling back the $\mathbb{CP}^1$-bundle by this map exhibits ruled surfaces have universal cover $\mathbb{CP}^1\times\widetilde{C}$ -- since $\widetilde{C}$ is Stein, we have $H^1(\widetilde{C}, \mathcal{PGL}(2, \mathbb{C})) = 0$ and hence the $\mathbb{CP}^1$ bundle over $\widetilde{C}$ is trivial.

A non-K\"ahler surface with Kodaira dimension $-\infty$ is called a \textit{Class $\mathrm{VII}$ surface}. A minimal such surface is called a \textit{Class $\mathrm{VII}_0$ surface}, and if furthermore the second Betti number is positive, then it is called a \textit{Class $\mathrm{VII}_0^+$ surface}. A Class $\mathrm{VII}_0$ surface with second Betti number zero is biholomorphic to a Hopf surface or an Inoue-Bombieri surface, see \cite{Bog1}, \cite{Bog2}, \cite{LYZ}, and \cite{Tel1}. Hopf surfaces have universal cover $\mathbb{C}^2\setminus\{0\}$ which is not contractible, while Inoue-Bombieri surfaces have universal cover $\mathbb{C}\times\mathbb{H}$ which is contractible, so they are aspherical. 

A \textit{spherical shell} in a complex surface is an open subset biholomorphic to a neighbourhood of $S^3$ in $\mathbb{C}^2\setminus\{0\}$. If the complement is connected, then it is called a \textit{global spherical shell}. A surface which admits a global spherical shell is a deformation of a primary Hopf surface\footnote{A Hopf surface $X$ is called \textit{primary} if $\pi_1(X) \cong \mathbb{Z}$. Such surfaces are diffeomorphic to $S^1\times S^3$.} blownup at finitely many points \cite[Theorem 1]{Kato} -- note that such surfaces are not aspherical. The global spherical shell conjecture asserts that all Class $\mathrm{VII}_0^+$ surfaces contains a global spherical shell. The conjecture remains open with some progress for small values of $b_2$, see \cite{Tel2}, \cite{Tel3}, \cite{Tel4}. It is not yet known if there exists an aspherical Class $\mathrm{VII}_0^+$ surface (it would necessarily violate the global spherical shell conjecture).

Since class $\mathrm{VII}$ surfaces have first Betti number $1$, such surfaces have $\chi_{top}(X) = b_2(X)$. Furthermore, as they are non-K\"ahler, we see that $b^+(X) = 2h^{2,0}(X) = 0$ and hence $\sigma(X) = -b^-(X) = -b_2(X)$. So Inoue-Bombieri surfaces have $\chi_{top}(X) = \sigma(X) = 0$, while aspherical class $\mathrm{VII}_0^+$ surfaces have $\chi_{top}(X) = -\sigma(X) = b_2(X) > 0$.

{\bf Kodaira dimension $0$:} In the K\"ahler case, there are two families: tori and their quotients (bi-elliptic surfaces), and K3 surfaces and their quotients (Enriques surfaces). The former have universal cover $\mathbb{C}^2$ and are therefore aspherical, while the latter have K3 surfaces as their universal cover and hence are not aspherical. 

In the non-K\"ahler realm, such surfaces are primary Kodaira surfaces and their quotients (secondary Kodaira surfaces). Primary Kodaira surfaces are holomorphic principal elliptic curve bundles over a smooth connected genus one curve. It follows from the long exact sequence in homotopy that $\pi_k = 0$ for $k > 1$. Just as in the case of ruled surfaces, we can also identify the universal cover of Kodaira surfaces as $\mathbb{C}^2$ by pulling back such a bundle by the universal covering of the base. A description of primary Kodaira surfaces as quotients of $\mathbb{C}^2$ by a group of affine transformations was given by Suwa \cite[Theorem 2]{Suwa75}.

{\bf Kodaira dimension $1$:} A compact surface $X$ is called an \textit{elliptic surface} if there is a smooth connected curve $C$ and a holomorphic map $\pi : X \to C$ such that the generic fiber is a smooth genus one curve; the map $\pi$ is called an \textit{elliptic fibration}. We call an elliptic surface $X$ \textit{relatively minimal} if there are no $-1$ curves in the fibers of $\pi$ -- every elliptic surface is an iterated blowup of a relatively minimal elliptic surface. Every surface of Kodaira dimension $1$ is elliptic (see \cite[Lemma 7.2(a)]{Wall86}), but there are also elliptic surfaces of Kodaira dimension $-\infty$ and Kodaira dimension $0$. An elliptic surface with Kodaira dimension $1$ is called a \textit{properly elliptic surface}. 

The non-generic fibers of a relatively minimal elliptic fibration $\pi : X \to C$, called exceptional fibers, were classified by Kodaira, see \cite[Theorem 6.2]{Kod63}. Aside from multiples of a smooth genus one curve (known as a multiple fibers with smooth reduction), every other possibility is a configuration of (possibly singular) rational curves. The elliptic fibration induces an orbifold structure on $C$ by declaring images of multiple fibers as cone points whose order is the multiplicity of the fiber. We denote the orbifold Euler characteristic and orbifold fundamental group of $C$ by $\chi^{\text{orb}}(C)$ and $\pi_1^{\text{orb}}(C)$ respectively.

\begin{prop}\label{asphericalelliptic}
An elliptic surface $X \to C$ is aspherical if and only if it is relatively minimal with no exceptional fibers other than multiple fibers with smooth reduction, and $X$ has Kodaira dimension $0$ or $1$.
\end{prop}
\begin{proof}
If $X$ is aspherical, then it is minimal (and hence relatively minimal) by \cite[Lemma 2]{ADC}. Furthermore, if $X$ is K\"ahler, then $X$ contains no rational curves, so the only exceptional fibers must be multiple fibers with smooth reduction. In the non-K\"ahler case, the same is true by \cite[Lemma 7.2(b)]{Wall86}. By \cite[Lemma I.3.18 (ii)]{FM}, we have $\chi(\mathcal{O}_X) = 0$. If $\chi^{\text{orb}}(C) > 0 = \chi(\mathcal{O}_X)$, then $\kappa(X) = -\infty$ by \cite[Lemma 7.1]{Wall86}. An aspherical surface with Kodaira dimension $-\infty$ is either Inoue-Bombieri or a Class $\mathrm{VII}_0^+$ surface. The former can't be elliptic as they contain no complex curves, and the latter can't be elliptic as they satisfy $c_1^2 < 0$. Therefore $\chi^{\text{orb}}(C) \leq 0$ and hence $X$ has Kodaira dimension $0$ or $1$ by \cite[Lemma 7.1]{Wall86}.

Conversely, if $X \to C$ is relatively minimal with no exceptional fibers other than multiple fibers with smooth reduction, and $X$ has Kodaira dimension $0$ or $1$, then $\chi^{\text{orb}}(C) \leq \chi(\mathcal{O}_X) = 0$. Therefore $C$ is a good orbifold, i.e. there is a finite orbifold covering $C' \to C$ where $C'$ is a manifold. Pulling back $X \to C$ by this map induces an elliptic fibration $X' \to C'$ with no multiple fibers such that $X'$ is a finite unramified cover of $X$, see \cite[Proposition III.9.1]{BHPV04}. Since $\chi(\mathcal{O}_{X'}) = 0$, all the fibres of $X' \to C'$ are isomorphic by \cite[Proposition V.12.2]{BHPV04} and the remark which precedes it, and hence $X' \to C'$ is locally trivial by \cite{FG65}. As the orbifold Euler characteristic is multiplicative under orbifold coverings, we have $\chi^{\text{orb}}(C') \leq 0$ and hence $C'$ has positive genus. Applying the long exact sequence of homotopy groups, we see that $X'$ is aspherical.
\end{proof}

\begin{rem}\label{bundle}
In the course of the proof, we showed that an aspherical elliptic surface $X$ has a finite\ cover $X'$ which is a holomorphic fiber bundle over a Riemann surface of positive genus, with an elliptic curve fiber. The structure group of this bundle is the automorphism group of the elliptic curve. Since translations form a finite index subgroup, there is a finite cover $C'' \to C'$ such that the pullback of $X' \to C'$ gives a bundle $X'' \to C''$ whose structure group reduces to the group of translations. That is, the bundle $X'' \to C''$ is a holomorphic principal elliptic bundle. Just as in the discussion of ruled surfaces, it follows that the universal cover of $X''$, and hence $X$, is biholomorphic to $\mathbb{C}\times\mathbb{C}$ if $X$ has Kodaira dimension $0$, or $\mathbb{C}\times\mathbb{H}$ if $X$ has Kodaira dimension $1$.\end{rem}

\begin{cor}
Aspherical elliptic surfaces contain no rational curves.
\end{cor}
\begin{proof}
If $X$ is an aspherical elliptic curve, then any map $\mathbb{CP}^1 \to X$ lifts to the universal cover since $\mathbb{CP}^1$ is simply connected. As described above, the universal cover is an open subset of $\mathbb{C}^2$, so the lift must be constant.
\end{proof}

There are examples of elliptic surfaces which contain finitely many rational curves, and examples with infinitely many, see \cite[section 5]{BFO23}.

Note, there are elliptic surfaces with Kodaira dimension $-\infty$, but none of them are aspherical (they are either rational, ruled, or Hopf). As for Kodaira dimension $2$, none of them are elliptic.

There are non-aspherical elliptic surfaces in Kodaira dimensions $0$ and $1$. By combining Proposition \ref{asphericalelliptic} with \cite[Lemma 7.2(b)]{Wall86}, such surfaces must be K\"ahler. In Kodaira dimension $0$, such surfaces are the elliptic K3 surfaces and all Enriques surfaces, while for Kodaira dimension $1$, homotopy K3 surfaces and Dolgachev surfaces provide examples. One can construct many more Kodaira dimension $1$ examples as follows (the stated examples arise this way). Choose an elliptic surface with an exceptional fiber which is not a multiple fiber with smooth reduction (equivalently, has positive Euler characteristic). Applying logarithmic transformations decreases the value of $\chi^{\text{orb}}(C)$, so by \cite[Lemma 7.1]{Wall86}, the result will eventually have Kodaira dimension 1.

{\bf Kodaira dimension $2$:} Aspherical surfaces with Kodaira dimension $2$ exist, but as with most problems regarding general type surfaces, we have nothing even close to a classification. Indeed, the list of known aspherical surfaces of general type is not particularly rich even if there are reasons to expect such surfaces exist in great profusion. The list includes ball quotients (e.g., fake projective planes), surfaces isogeneous to product of curves, Kodaira fibrations, Mostow-Siu surfaces, and certain branched covers of ball quotients due to Domingo-Stover \cite[Theorem 1.5]{S22}.  We refer to the paper of Bauer-Catanese \cite{Fabrizio} for more details.  The list of aspherical surfaces of general type also includes the vast majority of smooth minimal toroidal compactifications of ball quotients, see \cite[Theorem A]{tesi}. In all of these examples, when the signature is explicitly computed one has that $\sigma\geq 0$. It seems currently unknown whether or not an aspherical complex surface of general type can have negative signature. In conclusion, we can summarize this discussion into a table. 

\begin{table}[ht]
\centering
\resizebox{\textwidth}{!}{\begin{tabular}{c|c|c|c|c}
$\kappa(X)$ & $b_1(X)$ & {\bf List} & $\chi_{top}(X)$ & $\sigma(X)$\\
\hline
$-\infty$ & odd & \begin{tabular}{c}Inoue-Bombieri\\ Potential Class $\mathrm{VII}_0^+$ examples\end{tabular} & \begin{tabular}{c}$0$\\ Positive\end{tabular} & \begin{tabular}{c} $0$\\ Negative\end{tabular}\\
\hline
$0$ & \begin{tabular}{c}even\\ odd\end{tabular} & \begin{tabular}{c}Tori and quotients\\ All\end{tabular} & \begin{tabular}{c}$0$\\ $0$\end{tabular} & \begin{tabular}{c}$0$\\ $0$\end{tabular}\\
\hline
$1$ & \begin{tabular}{c}even\\ odd\end{tabular} & \begin{tabular}{c}Some\\ All\end{tabular} & \begin{tabular}{c}$0$\\ $0$\end{tabular} & \begin{tabular}{c}$0$\\ $0$\end{tabular}\\
\hline
$2$ & even & Some & Positive & ?
\end{tabular}}
\end{table}

\section{Singer Conjecture for Surfaces with Residually Finite Fundamental Group}\label{Proofs}

In this section, we show that the Singer conjecture holds true for closed aspherical complex surfaces with residually finite fundamental group. The proof we present here does not rely on Gromov's characterization of K\"ahler manifolds with non-vanishing first $L^2$-Betti number \cite{gromov}.  We rely upon the study of the Albanese map given in \cite{DCL19a} and \cite{DL23}, and on L\"uck's approximation theorem which we now briefly recall.

Let $X$ be a manifold with $\Gamma\stackrel{{\rm def}}{=}\pi_1(X)$  residually finite. We consider a sequence of nested, normal, finite index subgroups $\{\Gamma_i\}$ of $\Gamma$ such that $\cap_i\Gamma_i$ is the identity element. This  sequence is usually called a \emph{cofinal filtration} of $\Gamma$. Let $\pi_i\colon X_i\rightarrow X$ be the finite  regular cover of $X$ associated to $\Gamma_i$. L\"uck's approximation theorem \cite{Luck} ensures that
\begin{align}\label{LIMIT}
\lim_{i \to \infty}\frac{b_{k}(X_i)}{\deg \pi_i } \; = \; b^{(2)}_k (X; \widetilde{X}),
\end{align}
where $b_k(X_i)$ denotes the $k^{\text{th}}$ Betti number of $X_i$, and  $b^{(2)}_k (X; \widetilde{X})$ is the $L^2$-Betti number of $X$ computed with respect to the universal cover $\widetilde{X}$. Thus, the limit in \eqref{LIMIT} always exists and it is \emph{independent} of the cofinal filtration. We refer to the ratio $b_k(X_i)/\deg \pi_i$ as the \emph{normalized} $k^{\text{th}}$-Betti number of the cover $\pi_i\colon  X_i \rightarrow X$. Conjecture \ref{Singer} is then equivalent to the sub-degree growth of Betti numbers along a tower of covers associated to a cofinal filtration.

We start with the following proposition that is not limited to complex dimension two.

\begin{prop}\label{Albanese1}
	Let $X$ be an aspherical smooth projective variety. Assume that 
	 $\pi_1(X)$ is residually finite and   there exists a cofinal tower of coverings $\pi_i\colon X_i\to X$ 
	   such that the images $a_{X_i}(X_i)$  
	 of the Albanese maps 
	 are either  points or  curves in $\Alb(X_i)$. We then have
	\[
	\lim_{i \to \infty}\frac{b_{1}(X_i)}{\deg \pi_i } \; = \; 0.
	\]
\end{prop}
\begin{pf}
	Clearly, we just need to study the case where $b_{1}(X_i)\neq 0$ from some point on in the cofinal tower. 
	Recall that if  $a_{X_i}(X_i)$ is a curve, it must  be smooth, connected, and its genus equals 
	 $\frac{1}{2}b_{1}(X_i)$. For simplicity sake, from now on we assume that $a_{X_i}(X_i)$ is a curve for any $i\geq 0$ in the cofinal tower. Moreover, we set $S:=a_{X}(X)=a_{X_0}(X_0)$ and $S_i:=a_{X_i}(X_i)$. Due to the universal property of the Albanese variety, there is a map $a_{\pi_i} : S_i \to S$ such that the following diagram commutes:
	\begin{equation}\notag
	\centerline{ \xymatrix@=32pt{
			X_i\ar[d]_{\pi_i} \ar[r]^{a_{X_i}} & S_{i} \ar[d]^{a_{\pi_i}}  \\
			X  \ar[r]^{a_X}  & S.\\}} 
	\noindent 
	\end{equation} 
	Since $\pi_i$ is unramified, for any $i\geq 1$, the branching locus $B_i$ of $a_{\pi_i}$ is contained in the (finite) set of critical values of $a_{X}$.
 In particular, there exists a positive constant $C>0$ such that 
 for all $i$ we have  $ \#( B_i )  \leq C$.	
	Thus the degree of the ramification divisor $R_i$  of $a_{\pi_i}$ is bounded by
	$\deg R_i \leq C\cdot  \deg \big( a_{\pi_i} \big) $ for any $i$.
	 By using the Riemann--Hurwitz formula, we have
	\[
	b_1(S_i) \; = \;  2 \deg \big( a_{\pi_i} \big) \cdot \chi (\omega_S) + \deg R_i +2 
\;	\leq \;  \deg \big( a_{\pi_i} \big) \cdot  ( 2\chi(\omega_S) + C) +2.
	\] 
	Since $b_1(X_i)=b_1(S_i)$, dividing by $\deg \big( a_{\pi_i} \big) > 0$ yields
	\begin{equation}\label{eq:limitc}
	\frac{b_1(X_i)}{\deg(a_{\pi_i})} \leq 2\chi(\omega_S) + C + \frac{2}{\deg a_{\pi_i}} \leq 2\chi(\omega_S) + C + 2.
	\end{equation}
	
Next, 	let $k_i$ be the minimal degree of the restriction of $\pi_i$ to a general fiber   of $a_{X_i}$. 
%In fact, if the general fiber $F$ of $a_X$ has genus different from one, this is exactly the 
%degree of the restriction of $\pi_i$ to a general fiber of $a_{X_i}$.
Note that $\{k_i\}_{i\in\mathbb N}$ is a sequence of non-decreasing positive integers and  
\begin{equation}\label{degreerelation}
\deg \pi_i \geq k_i \, \cdot \, \deg \big( a_{\pi_i} \big) \quad \forall   i.
\end{equation}%\textcolor{red}{Include the argument for this inequality from email?}

	We claim that 
	\begin{equation}\label{limitk}
	\lim_{i\to\infty}k_i \; = \; \infty.
	\end{equation}
	By contradiction,  as the tower of coverings $\pi_i \colon  X_i\to X$ is cofinal, the $X_i$'s converge  to the topological universal cover $\widetilde{X}$   (\emph{cf.} \cite[\S3]{DD19}).
	Now equip the covers $X_i$ with the metrics induced by a fixed K\"ahler metric on the base, 
	given by an ample line bundle $L$ on $X$, via pullback.
Moreover, let $G_i$ be a general fiber of $a_{X_i}$ such that $k_i = \deg \big( \pi_i|G_i \big)$. In this way 
the volume of $G_i$  is computed as 
$$( \pi_i^*L \cdot G_i ) \; = \;  ( L \cdot \pi_{i*}G_i ) \; = \;  k_i \cdot ( L \cdot F) \quad \forall   i.$$
If the $k_i$'s were bounded,  there would exist an integer $N>0$ such that 
$( \pi_i^*L \cdot G_i ) < N$ for all $i$; but this contradicts 
 \cite[Proposition 3.3]{DD19} (note that if $X$ is aspherical, then $\pi_1(X)$ is large,
see   \cite[Proposition 6.7]{LMW17b} and \cite[Proposition 2.12.1]{Ko}).
%
%of $a_{X_i}$ would be bounded as they eventually stabilize. 
%	Indeed by the projection formula  
%	$$( \pi_i^*L \cdot G_i ) \; = \;  ( L \cdot \pi_{i*}G_i ) \; = \;  k_i \cdot ( L \cdot F).$$
%	Thus, by using the argument in Proposition 3.3 (and Lemma 3.4) of \cite{DD19},  
%	we can embed a general fiber of $a_{X_i}$ in the topological universal cover 
%	for $i$ sufficiently large. On the other hand, since $X$ is aspherical we have that 
%	$\pi_1(X)$ is large and therefore $\widetilde{X}$ cannot contain 
%	compact subvarieties of positive dimension 
%	(see  \cite[Proposition 6.7 ]{LMW17b}). 
	In conclusion, by \eqref{eq:limitc}, \eqref{degreerelation} and \eqref{limitk} it follows that
	\[
	0 \leq \frac{b_{1}(X_i)}{\deg \pi_i } \; \leq \; \frac{b_1(X_i)}{k_i\deg(a_{\pi_i})} \leq \frac{2\chi(\omega_S) + C + 2}{k_i }.
	\]
Taking the limit as $i \to \infty$ concludes the proof.
%
%The second case happens if $g(F)=1$ so that $X$ is an elliptic surface with non-negative Kodaira dimension fibered over a
%smooth curve of positive genus.
%Moreover $X$ is minimal as it is aspherical. It follows that the Kodaira dimension of $X$ must be one by the Easy Addition
%formula \cite[Corollary 2.3(iii)]{Mo}. Moreover the only singular fibers of $X$ are multiple fibers with smooth reduction
%again by the asphericity hypothesis. Therefore we have $\chi (\sO_X)=0$ by \cite[Lemma 3.18]{FM}.
%In this case, $X$ there exists a covering map $X'\to X$ from a fiber bundle of complex dimension two 
%and the result is then a consequence of   \cite[Theorem 4.1]{Luck1}.

%
%the natural map $\pi_1(F) \to \pi_1(X)$ is injective and 
%therefore $b_1^{(2)}(X) = 0$ 
%by a class	ical result of Cheeger--Gromov \cite[Corollary 0.6]{CG86}. This immediately proves
%the claimed limit by means of L\"uck's approximation.
% 

\end{pf}

We can now give a proof of Theorem \ref{main3} stated in the introduction.

\begin{pf}[Proof of Theorem \ref{main3}]
%The Kodaira--Enriques classification (see \cite[Table 10]{BHPV04}), can be used to reduce the validity of the Singer conjecture for aspherical surfaces to either non-K\"ahler surfaces (with  odd $b_1$), or  K\"ahler surfaces. 
Since $b_0^{(2)}(X; \widetilde{X}) = b_4^{(2)}(X; \widetilde{X}) = 0$ and the alternating sum of $L^2$-Betti numbers is equal to $\chi_{top}(X)$, it is enough to show that $b_1^{(2)}(X; \widetilde{X}) =0$ by Poincar\'e duality. 
Moreover, we can assume that $X$ is minimal by the asphericity assumption (see \cite[Lemma 2]{ADC}).
We divide the proof into several steps according 
to the Kodaira dimension (using the results of Section \ref{asphericalS}).

To begin with, suppose   ${\rm Kod}(X)=-\infty$.
In the non-K\"ahler case, $X$ is of class VII. All such surfaces have $b_1 (X) =1$. Moreover any finite   
covering of a surface of class VII is again of class VII (see for example \cite[Proposition II.1.21]{FM} or \cite[Lemma 5.1]{Dur05}).
 Thus, the vanishing of $b^{(2)}_{1}(X; \widetilde{X})$ follows immediately from L\"uck's approximation \cite{Luck}.
 In the K\"ahler  case, no surface with 
${\rm Kod}(X)=-\infty$ is aspherical.

%Next, if   ${\rm Kod}(X) = 0$ or ${\rm Kod}(X)=1$, 
%we observe that   $K3$ and Enriques surfaces have finite fundamental groups and as a result they are not aspherical. Moreover complex tori contain no rational curves.
%In the remaining cases, both in the K\"ahler and non-K\"ahler settings, 
%we suppose that  $X$ admits an elliptic fibration onto a curve. \
%By the following Lemma \ref{lem:elliptic}   and \cite[Chapter V.7]{BHPV04}, we observe that 
%  all singular fibers of $X$ have smooth reduction.
%Hence, by  \cite[Lemma I.3.18]{FM},  $\chi(\sO_X)=0$, and by 
%\cite[Theorem 2.9]{Dur05},
%the fundamental group of $X$   contains an infinite abelian  normal subgroup isomorphic to $\Z^2$.
%Hence, the $L^2$–Betti numbers vanish by a classical result of Cheeger--Gromov \cite[Corollary 0.6]{CG86},
%see also \cite[Theorem 1.44]{Luck02}. 

As discussed in Section \ref{asphericalS}, aspherical complex surfaces of Kodaira dimension $0$ are finitely covered by either a torus or a primary Kodaira surface. It follows that the fundamental groups of such surfaces contain a normal subgroup isomorphic to $\mathbb{Z}^2$. Hence, the $L^2$-Betti numbers vanish by a classical result of Cheeger-Gromov \cite[Corollary 0.6]{CG86}, see also \cite[Theorem 1.44]{Luck02}.

The aspherical complex surfaces of Kodaira dimension $1$ are the properly elliptic surfaces with no exceptional fibers other than multiple fibers with smooth reduction. By Remark \ref{bundle}, such surfaces are finitely covered by holomorphic elliptic curve bundle, and hence their fundamental groups also contain a normal subgroup isomorphic to $\mathbb{Z}^2$. Again, this implies that the $L^2$-Betti numbers vanish.

For surfaces of general type, we first recall that they have to be projective (see \cite[p.189]{BHPV04}),
 and we can   use the Albanese map if the surface is irregular (i.e. $b_1(X) \neq 0$).
 Then, we proceed by using L\"uck's approximation \cite{Luck} on a cofinal tower. 
 If none of the covers in the   tower is irregular, then the vanishing of $b^{(2)}_{1}(X; \widetilde{X})$ is immediate and the result follows. In the other cases, we use either Proposition 
 \ref{Albanese1}, or Theorem \cite[Theorem 1.3]{DL23} specialized to complex dimension two.  Recall that in complex dimension two, $a_{X}$ is semismall if and only if it is generically finite onto its image.
\end{pf}

\section{On the Singer Conjecture for Complex Surfaces}\label{SingerS}

In \cite{gromov}, Gromov shows that if $(X, \omega)$ is a closed, K\"ahler manifold with $b^{(2)}_{1}(X;\widetilde{X})\neq 0$, then $\pi_1(X)$ is \emph{commensurable} to the fundamental group of a compact surface of genus $g\geq 2$. For more details about this important result, we refer to \cite{abr} and the nice book \cite[Chapter 4]{Toledo} on K\"ahler groups.

Gromov's theorem implies that no aspherical K\"ahler surface $(X^2, \omega)$ can have non-vanishing $b^{(2)}_1$. Indeed, if this was the case then a finite cover of $X$, say $X^\prime$, would have the same fundamental group as a hyperbolic Riemann surface, say $C$. Since both $X^\prime$ and $C$ are aspherical with isomorphic fundamental groups, they are homotopy equivalent \cite[Theorem 2.1]{Luck2}, which is clearly not possible as $H_{4}(X';\Z)\neq H_{4}(C;\Z)=0$. Let's summarize this discussion into a theorem.

\begin{thm}[Gromov]\label{ksurfaces}
The Singer conjecture is true for closed, aspherical, K\"ahler surfaces.
\end{thm}

%\marginpar{\textcolor{red}{Should we capitalise `conjecture'?}}

We can now combine some parts of the proof of Theorem \ref{main3} with Theorem \ref{ksurfaces} to prove the following.

\begin{thm}\label{csurfaces}
The Singer conjecture is true for closed, aspherical, complex surfaces that are not in Class $\mathrm{VII}_0^+$.
\end{thm}
\begin{pf}
By Theorem \ref{ksurfaces}, we only need to check the non-K\"ahler case. First note that all Inoue-Bombieri surfaces have solvable fundamental group -- in fact, they are all solvmanifolds, see \cite[Theorem 1]{Has}. Since solvable groups are amenable, it follows that the $L^2$-Betti numbers vanish \cite[Corollary 6.75]{Luck02}. This leaves only the minimal non-K\"ahler surfaces of Kodaira dimension $0$ and $1$, all of which are elliptic. As in the proof of Theorem \ref{main3}, it follows from Remark \ref{bundle} that the fundamental group of such a surface contains a normal subgroup isomorphic to $\mathbb{Z}^2$. In this case, the $L^2$-Betti numbers vanish by a classical result of Cheeger-Gromov \cite[Corollary 0.6]{CG86}, see also L\"uck's \cite[Theorem 4.1]{Luck1} and \cite[Theorem 1.44]{Luck02}.
\end{pf}
%To begin with suppose   
%
%${\rm Kod}(X)=-\infty$. In the non-K\"ahler case, $X$ is of class VII and so it is excluded by the statement of this theorem. \normalmarginpar\marginpar{\raggedright\textcolor{red}{Need to address Inoue surfaces.}}
% In the K\"ahler  case, no surface with ${\rm Kod}(X)=-\infty$ is aspherical.
%
%If ${\rm Kod}(X) = 0$, the result is clear. If ${\rm Kod}(X)=1$, because of Theorem \ref{ksurfaces} we need to discuss the non-K\"ahler case only. By Lemma \ref{lem:elliptic} and as in the proof of Theorem \ref{main3}, if $X$ is aspherical and  
%it admits an elliptic fibration onto a curve, then $\pi_1(X)$ contains an infinite abelian  normal subgroup isomorphic to $\Z^2$. In this case, the $L^2$–Betti numbers vanish by a classical result of Cheeger--Gromov \cite[Corollary 0.6]{CG86},
%see also \cite[Theorem 1.44]{Luck02}. 
%
%For surfaces of general type, we first recall that they have to be projective (see \cite[p.189]{BHPV04}),
% and we can simply use Theorem \ref{ksurfaces}. This concludes the proof.
%\end{pf}

Note, if one could show that the fundamental group of a Class $\mathrm{VII}_0^+$ surface was residually finite, then we could apply the argument in the proof of Theorem \ref{main3} to extend Theorem \ref{csurfaces} to all complex surfaces.

The big elephant in the room. Aspherical Class $\mathrm{VII}_0^+$ surfaces conjecturally do not exist. That said, their cohomological structure seems somewhat simple. This motivates the following.

\begin{qtn}
Assume there are aspherical Class $\mathrm{VII}^+_0$ surfaces. Can we prove the Singer conjecture holds for them?
\end{qtn}

%In Theorem \ref{main3}, we address this question under the assumption that the fundamental group of a Class $\mathrm{VII}_0^+$ is residually finite.

\section{Reid's Conjecture and Gromov-L\"uck Inequality}

In this section, we prove Theorem \ref{Ginequality}.

%By the discussion in section 2, the only aspherical surfaces with Kodaira dimension $-\infty$ are Inoue-Bombieri surfaces and potential aspherical class $\mathrm{VII}_0^+$ surfaces. Since class $\mathrm{VII}$ surfaces have first Betti number $1$, such surfaces have $\chi_{top}(X) = b_2(X)$. Furthermore, as they are non-K\"ahler, we see that $b^+(X) = 2h^{2,0}(X) = 0$ and hence $\sigma(X) = -b^-(X) = -b_2(X)$. So Inoue-Bombieri surfaces have $\chi_{top}(X) = \sigma(X) = 0$, while aspherical class $\mathrm{VII}_0^+$ surfaces have $\chi_{top}(X) = -\sigma(X) > 0$. On the other hand, the aspherical surfaces of Kodaira dimension $0$ or $1$ all have $\chi_{top}(X) = \sigma(X) = 0$. This only leaves surfaces of general type.

By the discussion in Section \ref{asphericalS}, the only aspherical surfaces with Kodaira dimension $-\infty$ are Inoue-Bombieri surfaces and potential aspherical class $\mathrm{VII}_0^+$ surfaces. As we have seen, the former satisfy $\chi_{top}(X) = \sigma(X) = 0$, while the latter satisfy $\chi_{top}(X) = -\sigma(X) = b_2(X) > 0$. On the other hand, the aspherical surfaces of Kodaira dimension $0$ or $1$ all have $\chi_{top}(X) = \sigma(X) = 0$. This only leaves surfaces of general type.

Note, the Bogomolov-Miyaoka-Yau inequality states that for a general type surface $X$ we have $\chi_{top}(X) \geq 3\sigma(X)$. However, it is not true that $\chi_{top}(X) \geq 3|\sigma(X)|$ for every such $X$. For example, let $X_d$ be a smooth degree $d$ hypersurface of $\mathbb{CP}^3$. Note that $X_d$ is a surface of general type for $d \geq 5$, and a simple characteristic class argument shows that $\chi_{top}(X_d) = d^3 - 4d^2 + 6d$ and $\sigma(X_d) = -\frac{1}{3}(d-2)d(d+2)$. So, for example, we have $\chi_{top}(X_5) = 55$ and $\sigma(X_5) = -35$ so $3|\sigma(X_5)| = 105 > 55 = \chi_{top}(X_5)$. In fact, the proposed inequality is violated by $X_d$ for all $d \geq 5$ (also $d = 3, 4$, but these are not surfaces of general type). Of course, none of these examples are aspherical since they are simply connected by the Lefschetz hyperplane theorem (see for example \cite[Theorem 3.1.17]{Laz}).

If the signature is non-negative, then of course $\chi_{top}(X) \geq 3\sigma(X)$ is equivalent to $\chi_{top}(X) \geq 3|\sigma(X)|$. The discrepancy occurs, as in the examples above, when the signature is negative.

\begin{qtn}\label{negsign} Does there exist an aspherical complex surface with negative signature? (Either has to be a counterexample to the global spherical shell conjecture or a surface of general type).\end{qtn}

This question is yet to be answered, so we continue on our quest to find an inequality relating $\chi_{top}(X)$ and $|\sigma(X)|$. To do so, we need to recall the circle of ideas related to Reid's conjecture, see for example \cite[Chapter VII]{BHPV04}. We also refer to the beautiful survey \cite{LP12} of Mendes Lopes-Pardini on the geography of irregular surfaces for much more on this fascinating topic.

\begin{conj}[Reid]\label{Reid}
Let $X$ be a minimal surface of general type such that $K^{2}_{X}<4\chi_{hol}$, where $\chi_{hol}$ is the holomorphic Euler characteristic. Then $\pi_1(X)$ is either finite, or it is commensurable with the fundamental group of a curve.
\end{conj}

As shown by Horikawa \cite{Hor76}  and Reid \cite{Rei79}, Conjecture \ref{Reid} holds true under the stronger assumption $K^{2}_{X}<3\chi_{hol}$. We therefore can observe the following.

\begin{prop}\label{areid}
Let $X$ be an aspherical surface of general type. We then have $K^{2}_{X}\geq 3\chi_{hol}$.
\end{prop}

\begin{pf}
By \cite[Lemma 2]{ADC}, $X$ must be minimal. Now an aspherical surface must have infinite $\pi_1$, see for example \cite[Lemma 4.1]{Luck2}. As was observed at the beginning of Section \ref{SingerS}, $\pi_1(X)$ cannot be commensurable with the fundamental group of a curve. Since Conjecture \ref{Reid} holds true for minimal surfaces of general type satisfying $K^{2}_{X}<3\chi_{hol}$, we conclude that
\[
K^{2}_{X}\geq 3\chi_{hol},
\]
for any aspherical surface of general type.
\end{pf}

We can now prove the desired inequality relating $\chi_{top}(X)$ and $|\sigma(X)|$ for aspherical general type surfaces.

\begin{lem}
Let $X$ be an aspherical surface of general type. We then have:
\[
\chi_{top}(X)\geq \frac{9}{5}|\sigma|.
\]
\end{lem}

\begin{pf}
Recall that
\[
K^{2}_{X}=2\chi_{top}(X)+3\sigma(X), \quad \chi_{hol}(X)=\frac{\chi_{top}(X)+\sigma(X)}{4}.
\]
By using Proposition \ref{areid}, we obtain
\[
\chi_{top}(X)\geq \frac{9}{5}(-\sigma(X)),
\]
which, combined with the Bogomolov-Miyaoka-Yau inequality, gives $\chi_{top}(X)\geq \frac{9}{5}|\sigma(X)|$.
\end{pf}

This completes the proof of Theorem \ref{Ginequality}.

Note that for a minimal surface of general type $X$, we have $c_1^2(X) > 0$ from which it follows that $\chi_{top}(X) > \frac{3}{2}(-\sigma(X))$. If the signature is negative, the inequality $\chi_{top}(X) \geq \frac{9}{5}(-\sigma(X))$ is stronger. For example, if $\sigma(X) = -3$, the former inequality implies $\chi_{top}(X) \geq 5$ while the latter implies $\chi_{top}(X) \geq 6$.

\begin{rem}Reid's conjecture implies the slightly better bound $\chi_{top}(X) \geq 2|\sigma(X)|$.\end{rem}

Note that we actually have $\chi_{top}(X) \geq \frac{9}{5}|\sigma(X)|$ for all aspherical complex surfaces, except any potential Class $\mathrm{VII}_0^+$ examples (in all other cases, $\chi_{top}(X) = \sigma(X) = 0$).

\begin{rem}
Very recently, in \cite[Conjecture 1.2]{AMW}, Arapura, Maxim and Wang stated a Hodge-theoretic version of the Singer-Hopf conjecture: \emph{If $X$ is a compact K\"ahler manifold of dimension 
$n$ which is either aspherical or it has nef cotangent bundle, then} 
\begin{equation}\label{eq:AMW}
(-1)^{n-p} \chi(\Omega_X^p)  \geq 0\quad \quad  \mbox{\emph{for \, all}} \quad p=0, \ldots ,n.
\end{equation}
Here $\Omega_X^p$ denotes the bundle of holomorphic $p$-forms, and
 $$\chi(\Omega_X^p)=\sum_{i=0}^{n} (-1)^i \dim H^i(X, \Omega_X^p)$$ is the associated Euler characteristic.
The conjecture is verified by the same authors in the case of   surfaces with nef cotangent bundle (\emph{cf. loc. cit.} Proposition 2.4). Moreover, 
by following  \cite{JK93}, it also holds for aspherical complex surfaces as 
$$\chi(\omega_X) =\chi(\sO_X) = \frac{\chi_{ top}(X) + \sigma(X)}{4} \quad \mbox{and}\quad 
\chi(\Omega^1_X)=\frac{\sigma(X)-\chi_{top}(X)}{2}.$$ In higher dimension, 
the conjecture holds for K\"ahler hyperbolic and K\"ahler nonelliptic manifolds (see \cite{Gro}
and \cite{JZ}). As an application of  the inequality $\chi_{top}(X) \geq \frac{9}{5} |\sigma(X)|$ of 
Theorem \ref{Ginequality}, we observe that the inequalities in \eqref{eq:AMW} are actually strict for all 
complex aspherical surfaces of general type. More precisely, as aspherical surfaces are minimal and $\chi_{top}(X)>0$ for all minimal surfaces of general type, if $\sigma(X)\neq 0$ we have
\begin{align}\notag
\chi(\omega_X) =\chi(\sO_X) \geq \frac{1}{5}|\sigma(X)|>0 \quad \mbox{and}\quad 
-\chi(\Omega^1_X)\geq \frac{2}{5}|\sigma(X)|>0,
\end{align}
while if $\sigma(X)=0$ we clearly obtain
\begin{align}\notag
\chi(\omega_X) =\chi(\sO_X) = \frac{\chi_{ top}(X)}{4}>0 \quad \mbox{and}\quad 
-\chi(\Omega^1_X)=\frac{\chi_{top}(X)}{2}>0.\\ \notag
\end{align}
\end{rem}

It is tantalizing to ask what is the optimal constant $a>0$, such that $\chi_{top}(X)\geq a|\sigma(X)|$ for all aspherical surfaces of general type. As remarked above, we currently seem not to know any example of aspherical surfaces of general type with negative signature. If this is not an accident simply due to our lack of good examples, but a true fact of nature, by using the Bogomolov-Miyaoka-Yau inequality we would have
\[
\chi_{top}(X)\geq 3\sigma(X)\geq 0
\]
where the first inequality is saturated if and only if $X$ is a ball quotient. Notice that given a minimal surface of general type $X$ with $\sigma(X)> 0$, the reversed oriented $4$-manifold $\overline{X}$ can never admit a complex structure compatible with the orientation. This follows from Seiberg-Witten theory, see Theorem 2 in \cite{Kot}. Thus, in order to give a positive answer to Question \ref{negsign}, a genuinely new example of a surface of general type would need to be constructed, or alternatively one would need to provide an aspherical counterexample to the global spherical shell conjecture!

%\section{Data Availability Statement}

%No datasets were generated or analysed during the current study.

%\section{Conflict of Interests}

%On behalf of all authors, the corresponding author states that there is no conflict of interest. 

%\textcolor{red}{Is there a conjectural optimal constant? Given that there is a fixed orientation, maybe it makes more sense to search for two sided bounds $a|\sigma(X)| \leq \chi_{top}(X) \leq b|\sigma(X)|$? End with a comment along these lines?}

%
%%%%%%%%%%%%%%%%%%%%%
%\begin{thebibliography}{ELMNPM}
%%%%%%%%%%%%%%%%%%%%%
%
%    
%\bibitem[Sch]{Sch} Ch. Schnell. Notes on generic vanishing. \textit{Available on the author website}.
%
%\bibitem[Sto17]{Sto17} M. Stover, On general type surfaces with $q=1$ and $c_2=3p_g$.
%To appear in \textit{Manuscripta Math.}.
%
%\bibitem[Vid17]{Vid17} S. Vidussi, The Slope of Surfaces with Albanese Dimension One.
%To appear in \textit{Math. Proc. Cambridge Philos. Soc.}.
%
%\bibitem[Zha14]{Z} T. Zhang, Severi inequality for varieties of maximal Albanese dimension. \textit{Math. Ann.} \textbf{359} (2014), 1097--1114.
%
%\end{thebibliography}

\end{document}